\documentclass[reqno]{amsart}
\usepackage{amssymb}
\usepackage{mathtools}
\usepackage[arc, rotate]{xypic}
\usepackage[colorlinks=true,citecolor=magenta,linkcolor=blue]{hyperref}
\usepackage[capitalize]{cleveref}

\newcommand{\R}{\mathbb{R}}

\newcommand{\Aut}{\operatorname{Aut}}

\newcommand{\im}{\operatorname{im}\,}

\newcommand{\ad}{\operatorname{ad}}
\newcommand{\Ad}{\operatorname{Ad}}

\newcommand{\liealg}[1]{\mathfrak{#1}}

\newcommand{\rank}[1]{\operatorname{rank}{#1}}

\newcommand{\rest}[3]{{#1}_{#2\bigl.\bigr|#3}}

\newcommand{\liet}{\liealg{t}}
\newcommand{\lieg}{\liealg{g}}
\newcommand{\liek}{\liealg{k}}
\newcommand{\lieh}{\liealg{h}}

\newcommand{\inlinedef}[1]{\textit{#1}}

\theoremstyle{definition}

\newtheorem{theorem}{Theorem}[section]

\newtheorem{proposition}[theorem]{Proposition}

\newtheorem{corollary}[theorem]{Corollary}
\newtheorem{remark}[theorem]{Remark}
\newtheorem{example}[theorem]{Example}

\newcommand{\centeredalign}[1]
	{\begin{align*}#1\end{align*}}
						
\begin{document}
	\title[Equivariant formality of homogeneous spaces]{Equivariant formality of isotropy actions on homogeneous spaces defined by Lie group automorphisms}
	\author{Oliver Goertsches}
	\address{Mathematisches Institut der Universit\"at M\"unchen, Theresienstra\ss{}e 39, 80333 M\"unchen, Germany.}
	\email{goertsches@math.lmu.de}
	\author{Sam Haghshenas Noshari}
	\address{Mathematisches Institut der Universit\"at M\"unchen, Theresienstra\ss{}e 39, 80333 M\"unchen, Germany.}
	\email{noshari@math.lmu.de}
	\begin{abstract}
		We show that the isotropy action of a homogeneous space $G/K$, where $G$ and $K$ are compact, connected Lie groups and $K$ is defined by an automorphism on $G$, is equivariantly formal and that $(G, K)$ is
		a Cartan pair.
	\end{abstract}
	
	\maketitle	
	\section{Introduction}
An action of a compact Lie group $K$ on a manifold $M$ is called equivariantly formal if its equivariant cohomology $H^{\ast}_K(M)$ is a free module over the polynomial ring $S^{\ast}(\liek^{\ast})^{\Ad_K}$. Such actions are 
of interest for several reasons, most importantly because their equivariant cohomology algebra is relatively easily accessible for explicit computations: e.g., if $K$ is a torus, then the equivariant cohomology embeds, by 
Borel's Localization Theorem, into the equivariant cohomology of the fixed point set. In the following, we shall address the problem of determining whether or not the \inlinedef{isotropy action} on homogeneous spaces is
equivariantly formal. In more detail, we are interested in finding conditions for a compact, connected Lie group $K$ to act in an equivariantly formal fashion on the homogeneous space $G/K$ obtained from the compact 
and connected Lie group $G$. In this generality, the problem was, for example, investigated in \cite{shiga}, for the case of equal rank pairs (together with a description of the cohomology ring) in \cite{guilleminGKMDescription}, 
by the first author in \cite{goertsches} if $G/K$ is a symmetric space of the compact type and, more recently, in \cite{carlson} for spaces $G/K$ with $K$ a circular subgroup of $G$. Also here we shall focus on a 
special class of homogeneous spaces, namely those which are defined by Lie group automorphisms. For such we are able to prove:

\begin{theorem}
	\label{mt-theorem-main-theorem}
	Let $K$ be a Lie subgroup of the Lie group $G$ and suppose both $K$ and $G$ to be compact and connected. If there exists an automorphism on $G$ whose fixed point subalgebra coincides with the Lie algebra $\liek$ of
	$K$ then the pair $(G, K)$ is a Cartan pair and the isotropy action of $K$ on $G/K$ is equivariantly formal.
\end{theorem}

We refer the reader to \cref{mt-section-preliminaries} for a brief overview of the notion of Cartan pairs and equivariant formality. The proof of \cref{mt-theorem-main-theorem} is essentially performed in two steps. The first 
one, carried out in \cref{mt-section-proof-of-main-theorem-for-simple-lie-groups}, consists of a proof for those pairs $(G, K)$ for which $G$ is additionally assumed to be simple. The automorphisms on such Lie groups can be 
built by composing inner automorphisms and automorphisms of finite order, and the latter ones were classified by V. Ka\v{c} (see \cite[Section X.5]{helgason}). They give rise to so-called 
\inlinedef{generalized symmetric spaces}, for which formality in the sense of Sullivan (viz. being a Cartan pair) was verified by various authors, e.g. in \cite{kotschick} and \cite{stepien}. Using arguments from 
\cref{mt-section-equivariant-formality-of-isotropy-actions}, we conclude that in this case it suffices to consider pairs of Lie groups arising from automorphisms on the Dynkin diagram, and invoking, ultimately, the classification of 
these, we prove their equivariant formality in each case. Then, within the second step (see \cref{mt-section-proof-of-main-theorem}), we reduce the question of equivariant formality to the case considered in the first step, and 
conclude \cref{mt-theorem-main-theorem}.

\noindent\emph{Acknowledgements.} The following extends parts of the Master Thesis of the second named author, written at the University of Hamburg under the supervision of the first named author. We thank
Jeffrey Carlson for valuable discussions on a previous version and Aleksy Tralle for helpful remarks on generalized symmetric spaces and their formality.

	\section{Preliminaries}\label{mt-section-preliminaries}
\subsection{Equivariant formality}
Within the realms of compact manifolds $M$ admitting a (left-)action of a compact Lie group $K$, the \inlinedef{(real) equivariant cohomology} of $M$ can be defined to be the cohomology of the total complex obtained from
\inlinedef{Cartan's model}, a double complex whose degree $(p, q)$-term is given as
\centeredalign{
	\left(S^p(\liek^{\ast})\otimes\Omega^{q-p}(M)\right)^K.
}
Here, $S^p(\liek^{\ast})$ denotes the space of degree $p$ symmetric tensors, which we view as polynomials on the Lie algebra $\liek$ of $K$, and on which $K$ acts by the coadjoint representation; the action of $K$ on the
space $\Omega^{\ast}(M)$ of differential forms on $M$ is induced by pullback, and the superscript $K$ indicates we are only considering those elements in $S^p(\liek^{\ast})\otimes\Omega^{q-p}(M)$ which are invariant w.r.t. 
the action of $K$ induced by the respective actions on each of the factors. The action of $K$ on $M$ is then said to be \inlinedef{equivariantly formal} in case the spectral sequence obtained from the filtration of the Cartan
model by columns collapses at the $E^1$ stage, i.e. if the differential of the $r$-th page of the spectral sequence is trivial for all $r \geq 1$. Since in the sequel we will not be working with the actual definition, we refer 
the reader to \cite[Section 6.5]{guilleminSupersymmetry} for a discussion of equivariant cohomology in the Cartan model, especially its differentials, and to \cite[Appendix C]{guilleminMomentMaps} for the topological 
viewpoint. Here, we shall list some well-known equivalent characterizations of equivariant formality instead.	

\begin{proposition}[{\cite[Propositions C.24 and C.26]{guilleminMomentMaps}}]
	\label{mt-proposition-equivalent-characterizations-of-equivariant-formality}
	Let $K$ be a compact, connected Lie group acting on a compact manifold $M$ and $T \subseteq K$ a maximal torus. The following conditions are equivalent:
	\begin{enumerate}
		\item
		The action of $K$ on $M$ is equivariantly formal.
		
		\item
		The action of $T$ on $M$ is equivariantly formal.
		
		\item
		$\dim{H^{\ast}(M)} = \dim{H^{\ast}(M^T)}$.
	\end{enumerate}
\end{proposition}
		
Throughout, we will only be considering cohomology with real coefficients, and $M^T$ denotes the fixed point set of the action of $T$. We remark that under the hypothesis of 
\cref{mt-proposition-equivalent-characterizations-of-equivariant-formality} above, $\dim H^{\ast}(M^T) \leq \dim H^{\ast}(M)$ always holds, due to Borel's Localization Theorem, see \cite[p. 46]{hsiang}. 

\subsection{Fixed point sets of isotropy actions}\label{mt-subsection-fixed-point-sets}
Specializing to the case of a pair $(G, K)$ of compact and connected Lie groups such that $K \subseteq G$ is a subgroup, we can give a more explicit description of the fixed point set of the action of a maximal torus
in $K$ induced by the isotropy action.
\begin{proposition}[{\cite[Lemma 4.3]{goertsches}}]
	\label{mt-proposition-fixed-point-set-of-torus}
	Let $K \subseteq G$ be two compact, connected Lie groups and $T_K \subseteq K$ a maximal torus. Then the fixed point set
	$(G/K)^{T_K}$ of the action of $T_K$ on $G/K$ induced by the isotropy action of $K$ is homeomorphic to $N_G(T_K)/N_K(T_K)$.
	In particular 
	\centeredalign{
		\dim{H^{\ast}((G/K)^{T_K})} &= \dim{H^{\ast}(N_G(T_K)/N_K(T_K))}.
	}
\end{proposition}

Notice that the connected component containing the identity coset $p = eN_K(T_K)$ is exactly the orbit $Z_G(T_K)\cdot p$ of the canonical action of the centralizer $Z_G(T_K)$ on $N_G(T_K)/N_K(T_K)$. The isotropy subgroup of 
this action at $p$ is equal to $N_K(T_K) \cap Z_G(T_K) = Z_K(T_K)$, and because $T_K$ is a maximal torus in $K$, we have $Z_K(T_K) = T_K$. Thus, we have identified the identity component of $N_G(T_K)/N_K(T_K)$ as 
$Z_G(T_K)/T_K$. In view of \cref{mt-proposition-equivalent-characterizations-of-equivariant-formality}, we hence are left with determining the dimension of $H^{\ast}(Z_G(T_K)/T_K)$ and the number of components of 
$N_G(T_K)/N_K(T_K)$. The latter problem can be adressed using the following \begin{proposition}
	\label{mt-proposition-number-of-components-of-homogeneous-space}
	Let $K \subseteq G$ be two compact Lie groups. Then
	\centeredalign{
		\dim{H^0(G)} &= \dim{H^0(G/K)}\frac{\dim{H^0(K)}}{\dim{H^0(G_0\cap K)}};
	}
	the dimension of $H^0(G_0\cap K)$ is equal to the number of connected components of $K$ which are contained in the component $G_0$ of $G$ containing the identity element.
\end{proposition}

\subsection{Cohomology of homogeneous spaces}\label{mt-subsection-cohomology-homogeneous-spaces}
We still consider a pair of compact, connected Lie groups $(G, K)$ with $K$ a subgroup of $G$, and wish to collect from \cite{greub} some facts about the cohomology of the homogeneous space $G/K$. In 
$H^{+}(G) \subseteq H^{\ast}(G)$ we have the \inlinedef{primitive subspace} $P_G$, whose elements $\omega \in P_G$ are those on which the map 
$m^{\ast}:H^{\ast}(G)\otimes{}H^{\ast}(G) \cong H^{\ast}(G\times G) \rightarrow H^{\ast}(G)$ induced by the group multiplication $m:G\times G \rightarrow G$ takes the form 
\centeredalign{
	m^{\ast}(\omega) &= \omega\otimes{}1 + 1\otimes\omega.
}
Primitive elements have odd degree and constitute a graded subspace whose dimension is equal to the rank of $G$; moreover, the inclusion induces an algebra isomorphism $\wedge{}P_G \cong H^{\ast}(G)$ (cf.
\cite[Lemma VII and Theorem III, Section 5.18]{greub}). The cohomology algebra $H^{\ast}(G/K)$ then can be expressed as (cf. \cite[Theorem III, Section 11.5]{greub})
\centeredalign{
	H^{\ast}(G/K) &\cong A\otimes\wedge P_{(G, K)},
}
with $A \subseteq H^{\ast}(G/K)$ a graded subalgebra (dubbed \inlinedef{characteristic factor} in \cite{greub}) and $P_{(G, K)} \subseteq P_G$ the \inlinedef{Samelson subspace for the pair $(G, K)$}. This is the space 
consisting of all primitive elements which are representable by \inlinedef{$K$-basic forms}, that is, which lie in the image of the cohomology map $\pi^{\ast}:H^{\ast}(G/K) \rightarrow H^{\ast}(G)$ induced by the quotient map 
$\pi:G \rightarrow G/K$. In favorable situations, when $(G, K)$ is a \inlinedef{Cartan pair}, also the subalgebra $A$ can be identified explicitly, namely as the image of the Weil homomorphism 
$\omega_{(G, K)}:S^{\ast}(\liek)^{\Ad_K} \rightarrow H^{\ast}(G/K)$ associated to the principal $K$-bundle $\pi:G \rightarrow G/K$. By definition, $(G, K)$ is a Cartan pair if 
\centeredalign{ 
	\dim{P_K} + \dim{P_{(G, K)}} = \dim{P_G};
}
the left side is always bounded above by $\dim{P_G}$, see \cite[Theorem IV, Section 11.5]{greub}.

As is shown in \cite[Theorem IV, Section 11.5]{greub} as well, the notion of $(G, K)$ being a Cartan pair coincides with the more general notion of the space $G/K$ to be formal in the sense of Sullivan, and as such has been 
studied by various authors, cf. \cite{kotschick, stepien}. Our interest in Cartan pairs mostly stems from the fact that the dimension of their cohomology is rather well-understood:	
\begin{theorem}[{\cite[Theorem VI and Corollary, Section 10.15]{greub}}]
	\label{mt-theorem-dimension-of-image-of-weil-homomorphism}
	Suppose $(G, K)$ is a Cartan pair and denote by
	\centeredalign{
		\sum_{j=1}^{\rank{G}}t^{g_j}, \, \sum_{j=\rank{K} + 1}^{\rank{G}}t^{g_j} \text{ and } \sum_{j=1}^{\rank{K}}t^{l_j}
	}
	the Poincar\'e polynomials of $P_G$, $P_{(G, K)}$ and $P_K$, respectively. Then
	\centeredalign{
		\dim\im(\omega_{(G, K)}) &= \prod_{j=1}^{\rank{K}}\frac{g_j + 1}{l_j + 1}.
	}
\end{theorem}

For simple Lie groups $G$ the numbers $g_1, \ldots, g_{\rank{G}}$ are known as the \inlinedef{(primitive) exponents}, and the homogeneous primitive elements of degree $g_i$ correspond under transgression (see 
\cref{mt-theorem-transgression}) to generators of degree $(g_i + 1)/2$ of the polynomials invariant under the action of the Weyl group on a choice of a Cartan subalgebra for $\lieg$. In general, the 
product of the integers $(g_i + 1)/2$ is known to be equal to the order of  the Weyl group of $G$, cf. \cite[Corollary I, Section 11.7]{greub}. Thus, if $(G, K)$ is an equal rank pair, we retrieve the familiar formula 
$\dim{H^{\ast}(G/K)} = |W(G)|/|W(K)|$. 

We wish to highlight yet another class of Lie group pairs, the Cartan pairs for which the fiber inclusion $K \hookrightarrow G$ of the bundle $G \rightarrow G/K$ induces a surjection in cohomology: 
\begin{theorem}[{\cite[Theorems IX, Section 10.18, Theorem X, Section 10.19 and Theorem VI, Section 11.6]{greub}}]
	\label{mt-theorem-equivalent-characterizations-of-being-non-cohomologous-to-zero}
	Let $(G, K)$ be a pair of compact and connected Lie groups, $K \subseteq G$. The following conditions are equivalent:
	\begin{enumerate}
		\item
		$H^{\ast}(G) \rightarrow H^{\ast}(K)$ is surjective.
		
		\item
		The Weil homomorphism $\omega_{(G, K)}$ is trivial: $\omega_{(G, K)}(S^{+}(\liek^{\ast})^{\Ad_K}) = 0$.
		
		\item
		$\dim{H^{\ast}(G)} = \dim{H^{\ast}(G/K)}\cdot\dim{H^{\ast}(K)}$.
		
		\item
		The restriction map $S^{\ast}(\lieg^{\ast})^{\Ad_G} \rightarrow S^{\ast}(\liek^{\ast})^{\Ad_K}$ is surjective.
		
		\item
		There is an isomorphism of graded spaces $P_G \cong P_K\oplus{}P_{(G, K)}$.
	\end{enumerate}
\end{theorem}

Finally, we would like to point out that the cohomology of $G/K$ is entirely determined by the Lie algebra pair $(\lieg, \liek)$. In fact, the majority of results cited before were proven in \cite{greub} by transferring
the corresponding result in the space of (basic) alternating forms on $\lieg$ onto the group level via $\pi$. In particular, the spaces $P_G$, $P_K$ and $P_{(G, K)}$ have isomorphic counterparts which are intrinsically defined
on the level of Lie algebras (see \cite[Sections 5.18 and 10.4]{greub}), and $H^{\ast}(G/K)$ is isomorphic to the cohomology algebra $H^{\ast}_{bas}(\lieg)$ of $\liek$-basic forms on $\lieg$, cf. 
\cite[Proposition I, Section 11.1]{greub}. 
	
	\section{Equivariant formality of isotropy actions}\label{mt-section-equivariant-formality-of-isotropy-actions}
Combining the results about the fixed point set (\cref{mt-subsection-fixed-point-sets}) with those about the cohomology of homogeneous spaces allows us to rewrite the characterization of equivariant formality of the isotropy
action via the dimension of the fixed point set (\cref{mt-proposition-equivalent-characterizations-of-equivariant-formality}). In this section we shall show that this implies equivariant formality of the isotropy action to only
depend on the Lie algebra pair and how to relate the equivariant formality of subgroups of equal rank. These observations are motivated by the arguments found in \cite{goertsches} where the corresponding results were already
obtained for the case of symmetric spaces, cf. \cite[Proposition 4.6, Lemma 4.18 and Proposition 4.24]{goertsches}.
\begin{proposition}
	\label{mt-proposition-dimension-of-fixed-point-set-of-torus}
	Suppose $K \subseteq G$ are compact and connected Lie groups, and choose a maximal torus $T_K \subseteq K$. Then
	\centeredalign{
		\dim{H^{\ast}((G/K)^{T_K})} &= 2^{\rank{G} - \rank{K}}\cdot\dim{H^0((G/K)^{T_K})}
		\intertext{and}
		\dim{H^0((G/K)^{T_K})} &= \dim{H^0(N_G(T_K))}/|W(K)|,
	}
	where $W(K) = N_K(T_K)/T_K$ is the (analytic) Weyl group of $K$.
\end{proposition}	
\begin{proof}
	Recall (cf. the remark before \cref{mt-proposition-number-of-components-of-homogeneous-space}) that $N_G(T_K)_0 \cap N_K(T_K) = T_K$. By \cref{mt-proposition-number-of-components-of-homogeneous-space} 
	the number of components of $(G/K)^{T_K} = N_G(T_K)/N_K(T_K)$ (see \cref{mt-proposition-fixed-point-set-of-torus}) thus computes as
	\centeredalign{
		\dim{H^0((G/K)^{T_K})}	&= \dim{H^0(N_G(T_K))}/\dim{H^0(N_K(T_K))} \\
								&= \dim{H^0(N_G(T_K))}/|W(K)|, 
	}
	the last equation holding due to $W(K) = N_K(T_K)/T_K$ being a finite set and $T_K$ being connected. This shows the second equation. For the first, it suffices to observe that $T_K$ is entirely contained in the center of 
	$Z_G(T_K)$, so the identity component $Z_G(T_K)/T_K$ of the homogeneous space $N_G(T_K)/N_K(T_K)$ is again a Lie group of rank $\rank{G} - \rank{K}$ and $\dim{H^{\ast}(Z_G(T_K)/T_K)} = 2^{\rank{G} - \rank{K}}$.
\end{proof}		
\begin{corollary}
	\label{mt-corollary-equivariant-formality-only-depends-on-lie-algebra-pair}
	For $i \in \{1, 2\}$ let $K_i \subseteq G_i$ be compact and connected Lie groups. If there exists a Lie algebra isomorphism $\Phi:\lieg_1 \rightarrow \lieg_2$ with
	$\Phi(\liek_1) = \liek_2$ then the isotropy action of $K_1$ on $G_1/K_1$ is equivariantly formal if and only if the isotropy action of $K_2$ on $G_2/K_2$ is equivariantly formal.
\end{corollary}
\begin{proof}
	Let $T_{K_i} \subseteq K_i$ be maximal tori such that $\Phi(\liet_{\liek_1}) = \liet_{\liek_2}$. Then by 
	\cref{mt-proposition-equivalent-characterizations-of-equivariant-formality,mt-proposition-dimension-of-fixed-point-set-of-torus} the action of $K_i$ on $G_i/K_i$ is equivariantly formal iff
	\centeredalign{
		\tag{$\ast$}
		\dim{H^{\ast}(G_i/K_i)} &= 2^{\rank{G_i} - \rank{K_i}}\cdot\dim{H^0(N_{G_i}(T_{K_i}))}/|W(K_i)|.
	}
	Now, we have already remarked in \cref{mt-subsection-cohomology-homogeneous-spaces} the cohomology $H^{\ast}(G_i/K_i)$ to be isomorphic to the $\liek_i$-basic cohomology 
	$H^{\ast}_{\liek_i\text{-}bas}(\lieg_i)$, whose dimension is invariant under Lie algebra isomorphisms. Thus, the only factor in ($\ast$) that a priori might depend on the structure of $G_i$ and $K_i$ is 
	$\dim{H^0(N_{G_i}(T_{K_i}))}$. However, since $G_i$ is compact and connected, the exponential map of $G_i$ is surjective, and we thus obtain a homomorphism of groups
	\centeredalign{
		N_{G_i}(T_{K_i}) &\rightarrow H_i := \{ F \in \Aut(\liet_{\liek_i}) \mid \exists X \in \lieg_i: \, {e^{\ad_X}}_{|\liet_{\liek_i}} = F\}, \\
		\exp(X) &\mapsto \rest{\Ad}{\exp(X)}{\liet_{\liek_i}} = \rest{e^{\ad_X}}{}{\liet_{\liek_i}}.
	}
	This homomorphism is surjective and has as kernel $Z_{G_i}(T_{K_i})$, which is exactly the connected component of $N_{G_i}(T_{K_i})$ containing the identity. Therefore, 
	\cref{mt-proposition-number-of-components-of-homogeneous-space} implies that
	\centeredalign{
		\dim{H^0(N_{G_i}(T_{K_i}))} &= |N_{G_i}(T_{K_i})/Z_{G_i}(T_{K_i})| = |H_i|,
	}
	and because $\Phi$ is a Lie algebra isomorphism, the map $H_1 \rightarrow H_2$ given by
	\centeredalign{
		F = \rest{e^{\ad_X}}{}{\liet_{\liek_1}} & \mapsto \rest{e^{\ad_{\Phi(X)}}}{}{\liet_{\liek_2}} = \rest{\Phi}{}{\liet_{\liek_1}}\circ F\circ \rest{\Phi}{}{\liet_{\liek_1}}^{-1}
	}
	induces an isomorphism $H_1 \cong H_2$. In particular, $|H_1| = |H_2|$.
\end{proof}

\begin{corollary}
	\label{mt-corollary-characterization-of-being-non-cohomologous-to-zero-in-terms-of-equivariant-formality}
	Assume $K \subseteq G$ are two compact and connected Lie groups. The following are equivalent
	\begin{enumerate}
		\item
		\label{mt-proposition-characterization-of-being-non-cohomologous-to-zero-in-terms-of-equivariant-formality-item-1}
		The isotropy action of $K$ on $G/K$ is equivariantly formal and for every maximal torus $T \subseteq K$ the fixed point set $(G/K)^T$ is connected.
		
		\item
		\label{mt-proposition-characterization-of-being-non-cohomologous-to-zero-in-terms-of-equivariant-formality-item-2}
		The isotropy action of $K$ on $G/K$ is equivariantly formal and for some maximal torus $T \subseteq K$ the fixed point set $(G/K)^T$ is connected.
		
		\item
		\label{mt-proposition-characterization-of-being-non-cohomologous-to-zero-in-terms-of-equivariant-formality-item-3}
		$K$ is non-cohomologous to zero in $G$.
	\end{enumerate}
\end{corollary}
\begin{proof}
	If the first item is satisfied then so is the second. Now suppose that the second condition is fulfilled w.r.t. the maximal torus $T \subseteq K$. Because the isotropy action
	is equivariantly formal and $(G/K)^T$ is connected, \cref{mt-proposition-dimension-of-fixed-point-set-of-torus} shows that
	\centeredalign{
		\dim{H^{\ast}(G/K)} &= 2^{\rank{G} - \rank{K}} = \dim{H^{\ast}(G)}/\dim{H^{\ast}(K)},
	}
	and according to \cref{mt-theorem-equivalent-characterizations-of-being-non-cohomologous-to-zero} this is equivalent to $K$ being non-cohomologous to zero in $G$. Finally, assume that $K$ is non-cohomologous to zero and let 
	$T \subseteq K$ be any maximal torus. Because the dimension of the fixed point set of a torus action is always bounded by the dimension of the space being acted on 
	(cf. the remark after \cref{mt-proposition-equivalent-characterizations-of-equivariant-formality}), we have
	\centeredalign{
		\dim{H^0((G/K)^T)}\cdot 2^{\rank{G} - \rank{K}} &= \dim{H^{\ast}((G/K)^T)} \\
		&\leq \dim{H^{\ast}(G/K)} = 2^{\rank{G} - \rank{K}},
	}
	whence the isotropy action must be equivariantly formal and $(G/K)^T$ connected.
\end{proof}

Recall from \cref{mt-subsection-cohomology-homogeneous-spaces} that the Samelson subspace $P_{(G, K)}$ of a pair of compact and connected Lie groups $(G, K)$ comprises the primitive elements of $G$ which simultaneously 
are contained in the image of the induced map $H^{\ast}(G/K) \rightarrow H^{\ast}(G)$. Another description of this subspace is available using \inlinedef{universal transgression}, that is, transgression in the universal bundle over
the classifying space of $G$, see \cite[Sections 8 and 9]{borelCharacteristicClasses}. Translated to the level of Lie algebras, this corresponds to a map $\tau:P_G \rightarrow S^{\ast}(\lieg^{\ast})^{\Ad_G}$ which 
sends a homogeneous primitive element of degree $d$ to a polynomial of degree $(d + 1)/2$ and whose extension to a homomorphism of algebras $S^{\ast}(P_G) \rightarrow S^{\ast}(\lieg^{\ast})^{\Ad_G}$ is an isomorphism, cf. 
\cite[Theorem I, Section 6.13]{greub}. Moreover, we have:
\begin{theorem}[{\cite[Corollary II, Section 10.8]{greub}}]
	\label{mt-theorem-transgression}
	For every pair $(G, K)$ of compact and connected Lie groups with $K \subseteq G$, an element $\omega \in P_G$ is contained in the Samelson subspace if and only if $\tau(\omega)_{|\liek} = \imath^{\ast}(\tau(\omega))$ is 
	contained in $\im(\imath^{\ast}\circ\tau)\cdot S^{+}(\liek^{\ast})^{\Ad_K}$, where $\imath:\liek \hookrightarrow \lieg$ denotes the canonical inclusion and $\tau$ a transgression.
\end{theorem}

The existence of transgressions allows us to relate the equivariant formality of subgroups of equal rank:
\begin{proposition}
	\label{mt-proposition-relating-cohomology-of-equal-rank-pairs}
	Let $H \subseteq K \subseteq G$ be compact and connected Lie groups with $\rank{H} = \rank{K}$.
	\begin{enumerate}
		\item
		The Samelson subspace $P_{(G, K)}$ for the pair $(G, K)$ is equal to the Samelson subspace $P_{(G, H)}$ for the pair $(G, H)$.
		In particular, $(G, K)$ is a Cartan pair if and only if $(G, H)$ is a Cartan pair.
		
		\item
		If $(G, K)$ and $(G, H)$ are Cartan pairs then 
		\centeredalign{
			\dim H^{\ast}(G/K)\cdot\frac{|W(K)|}{|W(H)|} &= \dim H^{\ast}(G/H),
		}
		where $W(K)$ and $W(H)$ denote the (analytic) Weyl groups of $K$ and $H$, respectively. 
		
		\item
		\label{mt-proposition-relating-cohomology-of-equal-rank-pairs-item-eq-formality}
		If $(G, K)$ and $(G, H)$ are Cartan pairs then the isotropy action of $K$ on $G/K$ is equivariantly formal if and only if the isotropy action of $H$ on $G/H$ is equivariantly formal.
	\end{enumerate}
\end{proposition}
\begin{remark}
	\begin{enumerate}
		\item
		That $H$ is a Cartan pair if and only if $K$ is so is not a new result and was already noted in \cite[p. 212]{onishchik}. In fact, this observation was used in \cite{stepien} to prove formality, in the sense of Sullivan, of 
		generalized symmetric spaces.
		
		\item
		We thank Jeffrey Carlson for pointing out to us a different proof of \cref{mt-proposition-relating-cohomology-of-equal-rank-pairs-item-eq-formality} not mentioning the notion of Cartan pairs, see \cite[Theorem 1.1]{carlson}.
		Thus, the assumption of being a Cartan pair actually is dispensable but allows for a shorter proof and our exposition to remain self-contained.
	\end{enumerate}
\end{remark}
\begin{proof}
	\begin{enumerate}
		\item
		Since $H$ is a subgroup of $K$, the projection $G \rightarrow G/K$ factors through the canonical map $G/H \rightarrow G/K$. In particular, the image of $H^{\ast}(G/K) \rightarrow H^{\ast}(G)$ is contained in
		the image of $H^{\ast}(G/H) \rightarrow H^{\ast}(G)$, and therefore $P_{(G, K)} \subseteq P_{(G, H)}$. To prove the converse inclusion, fix a transgression $\tau:P_G \rightarrow S^{\ast}(\lieg^{\ast})^{\Ad_G}$
		(cf. \cref{mt-theorem-transgression}) and a maximal torus $\liet \subseteq \lieh$, which then also is a maximal torus in $\liek$ because $\liek$ and $\lieh$ are of equal rank. Then according to 
		\cref{mt-theorem-transgression}, an element $\omega \in P_G$ is contained in $P_{(G, K)}$ if and only if the restriction $\tau(\omega)_{|\liek}$ of the Polynomial $\tau(\omega)$ to $\liek$ can be written 
		as $\tau(\omega)_{|\liek} = f_{|\liek}\cdot g$ for Polynomials $f \in \im(\tau)$ and $g \in S^{+}(\liek^{\ast})^{\Ad_\liek}$. Thus, by virtue of Chevalley's Restriction Theorem 
		(\cite[Theorem 4.9.2]{varadarajan}), a primitive element $\omega$ is contained in $P_{(G, K)}$ if and only if $\tau(\omega)_{|\liet} = f_{|\liet}\cdot g$ for polynomials $f \in \im(\tau)$ and 
		$g \in S^{+}(\liet^{\ast})^{W(K)}$, where $W(K) = N_K(T)/T$ denotes the Weyl group of $K$ w.r.t. $T$. Now let $\omega \in P_{(G, H)}$. By the same reasoning we must have $\tau(\omega)_{|\liet} = f_{|\liet}\cdot g$ with 
		$f \in \im(\tau)$ and $g \in S^{+}(\liet^{\ast})^{W(H)}$. Notice that since  $\tau(\omega)$  and $f$ are polynomials in $\lieg$ their restrictions to $\liet$ must actually be invariant under $W(K)$, that is 
		$\tau(\omega)_{|\liet}, f_{|\liet} \in S^{\ast}(\liet^{\ast})^{W(K)}$. Hence if $kT \in W(K)$ is arbitrary then
		\centeredalign{
			f_{|\liet}\cdot g &= (kT).(f_{|\liet}\cdot g) = f_{|\liet}\cdot (kT.g),
		}
		and because $S^{\ast}(\liet^{\ast})$ is an integral domain this either forces $f_{|\liet}$ to vanish or $kT.g = g$ to hold. In either case it follows that $\omega \in P_{(G, K)}$. 
		
		\item
		Because $(G, K)$ and $(G, H)$ are Cartan pairs, we have
		\centeredalign{
			\dim{H^{\ast}(G/K)} &= \dim{\im(\omega_{(G, K)})}\cdot 2^{\rank{G} - \rank{K}},
		}
		and analogously
		\centeredalign{
			\dim{H^{\ast}(G/H)} &= \dim{\im(\omega_{(G, H)})}\cdot 2^{\rank{G} - \rank{H}}.
		}
		Here, $\omega_{(G, K)}$ and $\omega_{(G, H)}$ are the Weil homomorphisms of the pairs $(G, K)$ and $(G, H)$, respectively. Because of \cref{mt-theorem-dimension-of-image-of-weil-homomorphism}
		the dimension of the image of $\omega_{(G, K)}$ computes as
		\centeredalign{
			\dim{\im(\omega_{(G, K)})} &= \prod_{i=1}^{\rank{K}}\frac{g_i + 1}{l_i + 1} = \frac{1}{|W(K)|}\prod_{i=1}^{\rank{K}}\frac{g_i + 1}{2},
		}
		where 
		\centeredalign{
			\sum_{i=1}^{\rank{K}} t^{l_i}, \sum_{i=1}^{\rank{G}}t^{g_i} \text{ and } \sum_{i=\rank{K} + 1}^{\rank{G}}t^{g_i}
		}
		are the Poincar\'e polynomials of $P_K$, $P_G$ and the Samelson subspace of $(G, K)$, respectively. But since by the second part the Samelson subspaces of $(G, K)$ and $(G, H)$ coincide, we must also have
		\centeredalign{
			\dim{\im(\omega_{(G, K)})}\cdot\frac{|W(K)|}{|W(H)|} &= \dim{\im(\omega_{(G, H)})}.
		}
		
		\item
		Let $T \subseteq H \subseteq K$ be a maximal torus. By 
		\cref{mt-proposition-equivalent-characterizations-of-equivariant-formality,mt-proposition-dimension-of-fixed-point-set-of-torus} the isotropy action of $K$ on $G/K$ is equivariantly formal if and only if
		\centeredalign{
			\dim{H^{\ast}(G/K)} &= 2^{\rank{G} - \rank{K}}\dim{H^0(N_G(T))}/|W(K)|.
		}
		Multiplying both sides with $|W(K)|/|W(H)|$ and using the second part of this Proposition we see that this condition is equivalent to
		\centeredalign{
			\dim{H^{\ast}(G/H)} &= 2^{\rank{G} - \rank{K}}\dim{H^0(N_G(T))}/|W(H)|,
		}
		and, using \cref{mt-proposition-equivalent-characterizations-of-equivariant-formality} and \cref{mt-proposition-dimension-of-fixed-point-set-of-torus} again, 
		this equality is satisfied if and only if the isotropy action of $H$ on $G/H$ is equivariantly formal.
	\end{enumerate}
\end{proof}			

Given that most of our results involve the notion of Cartan pairs raises the question as to how these two notions of formality are related. As the following example shows, being a Cartan pair is not sufficient for the isotropy 
action to be equivariantly formal. 
\begin{example}
	Consider a compact and connected Lie group $G$ together with a $1$-dimensional subtorus $K$. Such pairs $(G, K)$ are always Cartan pairs as can be verified using \cref{mt-theorem-transgression}. To see this, fix a 
	transgression $\tau$ and, assuming the rank of $G$ to be at least $2$, choose a basis of homogeneous elements $\omega_1, \ldots, \omega_r \in P_G$. Let $f_{i|\liek}$ denote the restriction of $f_i = \tau(\omega_i)$ to 
	$\liek$ and suppose $\deg(f_1) \leq \deg(f_i)$. Then $f_{i|\liek} = f_{1|\liek}g$ for some polynomial $g$ on $\liek$ because $\liek$ is $1$-dimensional, so if $g$ is non-constant then by 
	\cref{mt-theorem-transgression} already $\omega_i \in P_{(G, K)}$ since every polynomial on $\liek$ is invariant. Otherwise, if $g$ is constant, say with image $c \in \R$, then $\omega_i - c\omega_1$ is a non-zero primitive 
	element whose image under $\tau$ restricts to zero in $S^{\ast}(\liek^{\ast})$ and hence is contained in $P_{(G, K)}$. Thus, $\dim{P_K} + \dim{P_{(G, K)}} = \dim{P_G}$.
	
	Pairs $(G, K)$ of this type were studied and the question of equivariant formality of the isotropy action settled in \cite{carlson}, see \cite[Algorithm 1.4]{carlson}. One of the non-equivariantly formal examples, which was 
	already given in \cite{shiga}, is the following: Let $S \subseteq SU(3)$ be the $1$-dimensional torus consisting of diagonal matrices with diagonal entries $e^{it}$, $e^{2it}$ and $e^{-3it}$, $t \in \R$. Then our previous 
	considerations show
	\centeredalign{
		\dim{H^{\ast}(SU(3)/S)} &= \dim{\im(\omega_{(SU(3), S)})}\cdot\dim\wedge P_{(SU(3), S)} = 2\cdot 2 = 4.
	}
	However, the normalizer $N_{SU(3)}(S)$ is connected, and so 
	\centeredalign{
		\dim{H^{\ast}((SU(3)/S)^S)} &= \frac{\dim{H^0(N_{SU(3)}(S))}}{|W_S(S)|}\cdot 2^{2-1} = 2 < 4.
	}
	It follows that the isotropy action of $S$ on $SU(3)/S$ can not be equivariantly formal.
\end{example}
	
	\section{The case of simple Lie groups}\label{mt-section-proof-of-main-theorem-for-simple-lie-groups}
As an intermediate step, we shall show that \cref{mt-theorem-main-theorem} is true when considering those Lie groups pairs $(G, K)$ for which $G$ additionally is simple. Choosing a maximal torus $\liet_{\lieg}$ in the simple 
Lie algebra $\lieg$ of $G$, it is a well-known fact (cf. \cite[Section X.5]{helgason} and \cite[Section 5]{wolf}) that an automorphism of the Dynkin diagram w.r.t. $\liet_{\lieg}$ (referred to as a symmetry of the 
Dynkin-Diagram in \cite{wolf}) induces a finite order automorphism $\tau$ on $\lieg$ and that the elements of the form $\tau\circ\Ad_g$, $g \in G$, exhaust all the automorphisms on $\lieg$ as $\tau$ ranges over all 
automorphisms of the Dynkin diagram. Moreover, the fixed point set $\lieh$ of such an automorphism of the Dynkin diagram has $\liet_{\lieh} = \liet_{\lieg} \cap \lieh$ as Cartan subalgebra, and if $\sigma$ is an automorphism 
on $\lieg$ with $\sigma \in \tau\Ad(G)$ then there exists an element $X \in \liet_{\lieh}$ and an inner automorphism $\alpha$ with
\centeredalign{
	\alpha\circ\sigma\circ\alpha^{-1} &= \tau\circ\Ad_{\exp(X)},
}
see \cite[Lemma 5.3]{wolf}. This in particular implies the fixed point set of $\tau\circ\Ad_{\exp(X)}$ to be mapped by $\alpha^{-1}$ to the fixed point set of $\sigma$, and since the condition of being a Cartan pair as well as 
the isotropy action to be equivariantly formal is invariant under Lie algebra isomorphisms (\cref{mt-corollary-equivariant-formality-only-depends-on-lie-algebra-pair}), it suffices to consider Lie group pairs $(G, K)$ arising 
from automorphisms $\sigma = \tau\circ\Ad_{\exp(X)}$ as above. In fact, we may even assume $\sigma = \tau$. Indeed, $\Ad_{\exp(X)}$ is the identity on $\liet_{\lieg}$, and by \cite[Lemma X.5.3]{helgason} the maximal torus 
$\liet_{\lieg}$ is equal to $Z_{\lieg}(\liet_{\lieh})$, the centralizer of $\liet_{\lieh}$ in $\lieg$. Since any other torus $\liet \supseteq \liet_{\lieh}$ thus is contained in $\liet_{\lieg}$ and $\sigma$ only fixes 
$\liet_{\lieh}$ on $\liet_{\lieg}$, $\liet_{\lieh}$ must be a maximal torus for both $\liek$ and $\lieh$; then \cref{mt-proposition-relating-cohomology-of-equal-rank-pairs} shows that $(G, K)$ is a Cartan pair whose isotropy action 
is equivariantly formal if and only if the same is true for the pair $(G, H)$ with $H \subseteq G$ the connected Lie subgroup corresponding to $\lieh$ (a similar argument, for automorphisms of finite order, was used in 
\cite{stepien} to prove formality of generalized symmetric spaces).

However, almost all such pairs $(G, H)$ arising from automorphisms on the Dynkin diagram are symmetric spaces, which were shown to be equivariantly formal in \cite{goertsches}, the single exception being pairs coming from
the triality automorphism on Lie algebras of type $D_4$. We can say even more: if $\sigma$ is built from an involution on the Dynkin diagram then by \cite[Lemma 5.3]{wolf} $\sigma$ satisfies the assumpions of 
\cite[Proposition 4.13]{goertsches}. Thus, the fixed point set in $G/H$ of the action of a maximal torus in $H$ must be connected by \cite[Proposition 4.14]{goertsches}, whence $H$ is even non-cohomologous to zero in $G$, cf. 
\cref{mt-corollary-characterization-of-being-non-cohomologous-to-zero-in-terms-of-equivariant-formality}. Since by \cite[p. 328]{borelLieCohomology} or \cite[Theorem 5]{kotschick} $H$ also is non-cohomologous to zero in $G$ for the 
pairs $(G, H) \cong (\operatorname{Spin}(8), G_2)$ being induced by the triality automorphism, we conclude \cref{mt-theorem-main-theorem} in the case of simple Lie groups $G$.
	
	\section{Proof of the main theorem}\label{mt-section-proof-of-main-theorem}
The remaining proof consists of a reduction of the general case considered in \cref{mt-theorem-main-theorem} to the special case dealt with in the previous section. Thus, we let $(G, K)$ be a pair of Lie groups as in 
\cref{mt-theorem-main-theorem} and write the Lie algebra of $G$ as 
\centeredalign{
	\lieg &= Z(\lieg)\oplus\bigoplus_{i=1}^r\lieg_i,
}
where $Z(\lieg)$ is the center and $[\lieg, \lieg] = \bigoplus_{i=1}^r\lieg_i$ is the sum of all simple ideals of $\lieg$. The automorphism $\sigma$ defining $K$ respects this decomposition in that its derivative $d\sigma$ 
maps $Z(\lieg)$ onto itself and permutes the set of simple ideals $\mathcal{I} = \{\lieg_1, \ldots, \lieg_r\}$. Hence the cyclic subgroup generated by $d\sigma$ partitions $\mathcal{I}$ into orbits 
$\mathcal{I}_1, \ldots, \mathcal{I}_k$, meaning that $d\sigma$ restricts to an automorphism on each such orbit $\liealg{o}_j = \bigoplus_{\liealg{i}\in\mathcal{I}_j}\liealg{i}$. The corresponding decomposition of the Lie 
algebra $\liek$ of $K$ then reads
\centeredalign{
	\liek &= \liek_0\oplus\liek_1\oplus\ldots\oplus\liek_k,
}
$\liek_0 = \liek\cap Z(\lieg)$ being the fixed point set of $d\sigma$ on $Z(\lieg)$, $\liek_j = \liealg{k}\cap\liealg{o}_j$ that on $\liealg{o}_j$. Let now $O_j$ be the connected subgroup of $G$ having Lie algebra $\liealg{o}_j$
and denote by $Z$ the connected component of the center of $G$. Next, also realize $K_0 \subseteq Z$, $K_j \subseteq O_j$ as the unique connected subgroups having Lie algebras $\liek_0$ and $\liek_j$, 
respectively. In the present situation, the multiplication map $m:Z\times{}O_1\times\ldots\times{}O_k \rightarrow G$, $(g_0, g_1, \ldots, g_k) \mapsto g_0g_1\ldots{}g_k$, is a
Lie group homomorphism inducing an isomorphism on the level of Lie algebras, and its derivative maps the Lie algebra of $K_0\times{}K_1\times\ldots\times{}K_k$ isomorphically onto $\liek$. In view of 
\cref{mt-corollary-equivariant-formality-only-depends-on-lie-algebra-pair}, the conclusions of \cref{mt-theorem-main-theorem} will thus be true for the pair $(G, K)$ if and only if they are true for the pair 
$(G', K') = (Z\times{}O_1\times{}\ldots\times{}O_k, K_0\times\ldots\times{}K_k)$. However, the canonical identification of $G'/K'$ with $M = Z/K_0\times{}O_1/K_1\times\ldots\times{}O_k/K_k$ is $K'$-equivariant with respect to 
the componentwise isotropy action of $K'$ on $M$, and thus the K\"unneth formula together with \cref{mt-proposition-equivalent-characterizations-of-equivariant-formality} shows that the action of $K'$ on $G'/K'$ is 
equivariantly formal if and only if each of the factors of $K'$ acts in an equivariantly formal fashion on the corresponding factor in $M$. 

Henceforth, our goal is to verify the equivariant formality of the isotropy actions of $K_0$ on $Z/K_0$ and each $K_j$ on $O_j/K_j$, $1 \leq j \leq k$. Now $Z$ is abelian, whence $K_0$ is even non-cohomologous to zero in $Z$, 
and each $O_j$ associated to an orbit $\liealg{o}_j$ of length one is the Lie group of one of the simple ideals of $\lieg$, which we already have dealt with in \cref{mt-section-proof-of-main-theorem-for-simple-lie-groups}.
Therefore, we may restrict our attention to those factors $O_j$ of $G'$ whose Lie algebras $\liealg{o}_j = \liealg{i}_1\oplus\ldots\oplus\liealg{i}_l$ consist of $l > 1$ summands, and fixing one such $\liealg{o}_j$, we may 
assume the summands to be enumerated such that $d\sigma^i(\liealg{i}_1) = \liealg{i}_{i+1}$ and $d\sigma^l(\liealg{i}_1) = \liealg{i}_1$. This arrangement implies the fixed point set $\liek_j$ on $\liealg{o}_j$ to be 
isomorphic to the fixed point set $\lieh$ of $d\sigma^l$ on $\liealg{i}_1$. More precisely, if we let $\Delta(\lieh) \subseteq \liealg{i}_1\oplus\ldots\oplus\liealg{i}_1$ denote the diagonal embedding of $\lieh$ into the 
$l$-fold direct sum of $\liealg{i}_1$ then the map $d\sigma^0\oplus\ldots\oplus{}d\sigma^{l-1}$ is an isomorphism  $\liealg{i}_1\oplus\ldots\oplus\liealg{i}_1 \rightarrow \liealg{i}_1\oplus\ldots\oplus\liealg{i}_l$ mapping 
$\Delta(\lieh)$ onto $\liek_j$. We thus may assume $O_j$ to be the $l$-fold product $O_j = I\times\ldots\times{}I$ of a compact, simple Lie group $I$ and $K_j = \Delta(H)$ to be the diagonal embedding of the connected 
component $H \subseteq I$ of the fixed point  set of an automorphism on $I$. As seen in \cref{mt-section-proof-of-main-theorem-for-simple-lie-groups}, we may find a maximal torus $T$ of $I$ such that $H$ shares the maximal 
torus $T \cap H$ with the fixed point set of an automorphism on $I$ arising from an automorphism on the Dynkin diagram of $I$, and therefore we may assume $H$ to even be defined by an automorphism of the Dynkin diagram right 
away. By \cref{mt-section-proof-of-main-theorem-for-simple-lie-groups}, such an $H$ not only is equivariantly  formal but even non-cohomologous to zero in $I$, hence if we can show that also $\Delta(I)$ is non-cohomologous to 
zero in $O_j$ then also $\Delta(H)$ must be, for the inclusion $\Delta(\lieh) \hookrightarrow \liealg{o}_j$ factors through the inclusion $\Delta(\liealg{i}) \hookrightarrow \liealg{o}_j$. But this is being guaranteed by 
\cref{mt-theorem-equivalent-characterizations-of-being-non-cohomologous-to-zero} and the fact that the composite map
\centeredalign{
	\xymatrix{
		\bigotimes_{i=1}^l S^{\ast}(\liealg{i}^{\ast})^{\Ad_I} \cong S^{\ast}(\liealg{o}_j^{\ast})^{\Ad_{O_j}} \ar[r]	& 
			S^{\ast}(\Delta(\liealg{i})^{\ast})^{\Ad_{\Delta(I)}} \ar[r]^{\Delta^{\ast}}_{\cong}	& S^{\ast}(\liealg{i}^{\ast})^{\Ad_I},
	}
}
in which the undecorated map is induced by the canonical inclusion $\Delta(\liealg{i}) \hookrightarrow \liealg{o}_j$, is a surjection. This finishes the proof of \cref{mt-theorem-main-theorem}.
	
 	\bibliographystyle{amsplain}
 	\bibliography{mt-bibliography}

\end{document}